\documentclass[11pt]{amsart}
\usepackage{amsmath}
\usepackage{amsthm,color,soul}
\usepackage{amssymb}
\usepackage[all]{xy}

\newtheorem{theorem}{Theorem}[section]
\newtheorem{proposition}[theorem]{Proposition}
\newtheorem{corollary}[theorem]{Corollary}
\newtheorem{lemma}[theorem]{Lemma}

\theoremstyle{definition}
\newtheorem{definition}[theorem]{Definition}

\newtheorem{remark}[theorem]{Remark}

\title[Extension of Lipschitz-type operators]
{Extension of Lipschitz-type operators on Banach function spaces}

\author[ W. V. Cavalcante, P. Rueda,  E. A. S\'{a}nchez-P\'erez]{W. V. Cavalcante, P. Rueda, E. A. S\'{a}nchez-P\'{e}rez$^*$}

\address{Wasthenny V. Cavalcante\\
Departamento de Matem\'atica\\ Federal University of Pernambuco\\Av. Prof. Moraes Rego, 1235 - Cidade Universit\'aria, Recife - PE, 50670-901. Brasil}

\email{wasthenny@dmat.ufpe.br}

\address{Pilar Rueda\\
	Departamento de An\'alisis Matem\'atico\\ Universidad de Valencia  \\C/ Dr. Moliner 50, 46100 Burjassot (Valencia). Spain}

\email{pilar.rueda@uv.es}

\address{Enrique A. S\'{a}nchez P\'{e}rez\\
Instituto
Universitario de Matem\'atica Pura y Aplicada,  Universitat
Polit\`ecnica de Val\`encia\\ Camino de Vera s/n, 46022 Valencia. Spain.}

\email{easancpe@mat.upv.es}



\thanks{$^*$ Corresponding author.}

\thanks{The first author gratefully acknowledges the support of the CAPES (Brazil).}

\thanks{The second and the third authors gratefully acknowledge the support of the Ministerio de Ciencia, Innovaci\'on y Universidades (Spain) and FEDER under grant   MTM2016-77054-C2-1-P2. }

\date{\today}

\begin{document}

\begin{abstract}
We study extension theorems for Lipschitz-type operators acting on metric spaces and with values on spaces of integrable functions.  Pointwise domination is not a natural feature of such spaces, and so almost everywhere inequalities and other measure-theoretic notions are  introduced.
We analyze  Lipschitz type inequalities in two fundamental cases. The first concerns a.e. pointwise inequalities, while the second considers dominations involving integrals. These Lipschitz type inequalities provide the suitable frame to work with operators that take values on Banach function spaces. In the last part of the paper we use some interpolation procedures to extend our study to interpolated Banach function spaces.

\end{abstract}

\maketitle


\section{Introduction and basic definitions}

Extension of Lipschitz functions acting on subsets of  metric spaces is a relevant issue in mathematical analysis, not only because of its theoretical interest but also because of the large number of applications that have been obtained.  There are two classical extension results that are considered as 
milestones in the theory. The first one  is the McShane-Whitney theorem, which concerns real valued functions, and establishes that given a subset
$S$  of a metric space $(M,d)$ and a Lipschitz function $T : S \to \mathbb  {R}$ with Lipschitz constant $k$, there is always a Lipschitz function $ M \to \mathbb {R}$  extending $T$ and with the same Lipschitz constant $k$. 
The second one is the celebrated Kirszbraun theorem, and deals with Lipschitz maps between Hilbert spaces. It asserts that for Hilbert spaces $H$ and $K,$ a subset $U \subseteq H$ and a Lipschitz operator
$T : U \to K,$ there is another  Lipschitz operator
$\hat{T}: H \to  K$
that extends $T$ with the same Lipschitz constant (see  \cite{kir}, \cite[p.21]{sch}). It is well-known that this result is not true in the class of  Banach spaces, not even in the finite dimensional case.
Other relevant aspects of the theory that are also of interest for our research have been developed recently, such as those concerning Lipschitz functions on ``$\infty$-type" spaces as $\ell^\infty,$ $c_0$ and $C(K)$-spaces. For example, the reader can find a remarkable complete study of Lipschitz operators on $C(K)$ spaces in a series of papers by Kalton (see \cite{kal,kal2,kal3} and the references therein; see also \cite{dogkallan}). After an inspection of these works, it can be noticed that the structure of these ``$\infty$-type" spaces ---formed  by  functions (and not by classes of a.e. equal functions) and with an $\infty$-type norm---, 
constitutes a fundamental part of the arguments for obtaining extensions of  Lipschitz operators on them.  
Although some of the ideas developed in this setting can be applied also for spaces of integrable functions, we will show that it is convenient to adapt them using suitable measure theoretical notions. 
Indeed, we will show that the development of extension procedures for function-valued maps forces the use of specific Lipschitz-type inequalities taking into account the nature of the function spaces in the range. Thus, if $(\Omega, \Sigma, \mu)$ is a measure space,
we will consider ``almost everywhere pointwise domination" for spaces as $L^\infty(\mu),$ as well as ``integral domination" or ``measure domination" for spaces as $L^1(\mu)$ or $L^p(\mu).$

Therefore, in this paper we  provide extension results for Lipschitz-type operators on metric spaces of measurable functions by adapting the notion of Lipschitz map to this kind of spaces when necessary. 
We introduce a new definition of a subclass of Lipschitz maps in which some elements associated to the underlying measure space appear explicitly. In particular, if  $S$ is a subset of the metric space $(M,d)$ and $Y(\mu)$ is a Banach function space, we consider
 Lipschitz operators  $T:M \to Y(\mu)$   satisfying  domination properties of the form
$$
\| \big( T (x) - T (y)  \big) \chi_A \|_{Y(\mu)} \le \phi(A) \, d(x, y) , \quad x, y \in X, \,\,\, A \in \Sigma,
$$
where $\phi: \Sigma \to \mathbb R$ is a set function related to the measure $\mu.$

Another aspect on factorization of Lipschitz maps that will be studied here is the maximality of the extensions. That is, given an extension of a Lipschitz map, what can we say about the size of the domain? Related to this question, it makes sense to ask if it
is still possible to find a maximal extension. Thus, we will also analyze if it is possible to find a bigger complete metric space $\hat M$ such that
$M$ is dense in $\hat M$ and $T$ is still Lipschitz as a map $\hat T: \hat M \to Y(\mu),$ under the requirement  that $\hat M$ is the bigger space satisfying this property.  
These results will be presented at the beginning of the paper (Section \ref{S2}), and will be used in the rest of the sections.  The advantage of these preliminary results concerning maximality of the extension is that they do not need any structure on the range spaces. As we will show, they can be established for general Lipschitz maps between metric spaces.

Thus, we will present our results in four sections. After the Introduction, 
 Section \ref{S2} will be devoted to show two theorems on the existence of extensions of Lipschitz operators between metric spaces and the description of the maximal domain  space where the extension can be defined. We
will present our specific results for
spaces of measurable functions in Sections \ref{S3} and \ref{S4}. In Section \ref{S3}, pointwise domination will be studied, showing that this property fits with order bounded Lipschitz-type operators, or with operators with values on $\infty$-type Banach  lattices ---$\ell^\infty$ and $C(K)$---. In Section \ref{S4} we will explain our extension theorems for  Lipschitz-type operators  on general Banach function spaces, in which  domination inequalities involving the measure will be considered. All these results will be complemented by the ones obtained in Section \ref{S2}, including results on maximal factorizations to the extensions as corollaries. In order to provide tools to use the results obtained in specific cases of Banach function spaces, in Section \ref{S5} we will show how to extend Lipschitz-type inequalities to larger classes of Banach function spaces in some concrete cases using simple interpolation arguments.

\vspace{1cm}

We will use standard Banach space notation throughout the paper. $(\Omega,\Sigma,\mu)$ will be a $\sigma$-finite measure space and $(E,\rho)$ a metric space.
A Banach function space (Bfs for short) $Y(\mu)$  over $\mu$ is an (linear) ideal of the space $L^0(\mu)$  (the linear space formed by classes of $\mu$-a.e. equal measurable functions), endowed with a complete norm $\|\phantom{f}\|_{Y(\mu)}$. That is, if $f \in L^0(\mu),$ $|f| \le g$ and $g \in Y(\mu),$
then $f \in Y(\mu)$   and $\|f\|_{Y(\mu)} \le \|g\|_{Y(\mu)}.$ It also contains  characteristic functions of finite measure sets. Sometimes we will write $Y$ instead of $Y(\mu)$ in case the measure $\mu$ is clearly fixed in the context. As usual, $f \vee g$ denotes the maximum of $f$ and $g.$  A Banach function space is order continuous if for any decreasing  sequence $f_n \downarrow 0$ we have that $\lim_n \|f_n\|_{Y(\mu)}=0$.
It is well-known that, if $Y(\mu)$ is order continuous,  the dual space $Y(\mu)^*$ coincides with  its  K\"othe dual
$$
Y(\mu)'= \Big\{ g \,\, \textit{is a class of measurable functions}: \, \int_\Omega f \, g \, d \mu < \infty, \, f \in Y(\mu)  \Big\}.
$$
Duality is then represented as the integral of the pointwise product of the functions involved. The reader can find more information  in 
\cite[p.28 ff]{lint} and \cite[Ch.2]{libro}.

Regarding metric spaces, throughout the paper $(M,d)$ will represent a (non-necessarily complete) metric space, and $S$ a subset of $M$. No further requirements will be assumed on $M$ unless they are requested explicitly. Recall that a map $T:S \to Y(\mu)$ is  
 $K$-Lipschitz (or Lipschitz with constant $K$) if
for every $x,y \in S,$
\begin{equation} \label{for1}
\|T(x)- T(y)\|_{Y(\mu)} \le \, K \, d(x,y),
\end{equation}
and $K$ is the smallest constant satisfying the inequality. 
As we are concerned with range spaces of the form $Y(\mu)$, we will work with variants of the above
classical domination inequality that fit with Bfs-valued functions in a sensible way.
Our reference for defining the explicit formulas for the extensions of  these Lipschitz-type functions is the McShane-Whitney theorem for real-valued maps which states that, if
$S$  is a subset   of a metric space $(M,d)$ and $T : S \to \mathbb  R$ is a Lipschitz function with constant $K$, there is always a Lipschitz function $\tilde T: M \to \mathbb R$  extending $T$ and with the same Lipschitz constant $K$. The function
\begin{equation}  \label{forext2}
{{T^M}}(x):=\sup_{u\in S}\{T(u)-K\,d(x,u)\}, \quad x \in M,
\end{equation}
provides such an extension, and it is sometimes called the McShane extension.  The Whitney formula, given by
\begin{equation}  \label{forext3}
 T^W(x) :=\inf_{u\in S}\{T(u)+K\,d(x,u)\}, \quad x \in M,
\end{equation}
gives also such  an extension.

We refer to the classical monograph \cite{We} (second  edition) and the recently appeared book \cite{Co} for the general theory of Lipschitz functions.


\section{Extension of Lipschitz functions between metric spaces to maximal metric domains} \label{S2}

In order to introduce the notion of maximality of the Lipschitz extensions that will be treated in the paper, let us start with two general results for maps between generic metric spaces. We will show that the concept of maximality  can be formulated for general metric spaces without assuming further properties. However, we will see in the next sections that these results can be improved when more requirements on the structure of the spaces involved are added. 

The first result ensures the existence of an extension, while the second one proves that under reasonable assumptions this extension is maximal, in the sense that its domain space is the biggest metric space with some special requirements to which the Lipschitz map can be extended. Although the results will be used in the next sections under many other restrictions, we consider here the more abstract context in which they work.

Let us introduce a technical definition that will be needed for our purposes. 
Let $M$ and $N$ be metric spaces. We say that a map $j:M \to N$ is an {\em  inclusion/quotient map} if $j(M)$ is dense in $N$ and there is an equivalence relation providing equivalence classes in $M$ such that the map $j(x)=j(y)$ for $x,y$ belonging to the same equivalence class. That is, $j$ is a quotient map on a dense subspace of $N$.

Let $(M,d)$ be a  metric space and  let $(E,\rho)$ be a complete metric space. Given  a $K$-Lipschitz map  $T:M \to E$,  consider the function $d_T: M\times M\to \mathbb R$ defined by 
$$
d_T(x,y) = \frac{1}{K} \rho (T(x),T(y))
$$
for any  $x,y\in M$. Clearly, $d_T$ is a pseudo metric. We can consider the equivalence classes associated to $d_T$ given by
\[
[x] = \{ y\in M: d_T(x,y)=0 \}, \quad x\in M.
\]
 The quotient space $M^{*} = \{[x]:x\in M\}$, formed by the equivalence classes, becomes a metric space when endowed with the distance $d^{*}_T([x],[y]) = d_T(x,y)$. Define the map $i^{*}:M\to M^{*}$ by $i^{*}(x) = [x]$. Since $T$ is $K$-Lipschitz, we have that $i^{*}$ is $1$-Lipschitz. Write $M_T$ for the completion of $(M^{*},d_T^*)$ and $\overline d_T$ for the extended metric on $M_T$. Note that the natural map $\overline i: M^{*}\to M_T$ is an isometric inclusion. Defining the inclusion/quotient map $j:M\to M_T$ by $j = \overline i \circ i^{*}$, we can easily ckeck that  $j$ is $1$-Lipschitz and $\overline {j(M)}^{\overline d_T}= M_T$.

Now, define $T^{*}: M^{*}\to E$ by $T^{*}([x]) = T(x)$. Note that $T^{*}$ it is well defined, since if $y,z\in [x]$, then $d_T(y,z) = 0$ and so $\rho(T(y),T(z))=0$. Let us see that $T^{*}$ is $K$-Lipschitz. Indeed,
\begin{eqnarray*}
\rho(T^{*}([x]),T^{*}([y]))&=& \rho(T(x),T(y))\\
&=& K \, d^{*}_T([x],[y]),\quad x,y\in X.
\end{eqnarray*}
We can extend $T^{*}$ to $\overline T:M_T\to E$ by continuity, providing the factorization $T = \overline T\circ j,$ and clearly the Lipschitz constant of $\overline T$ is still $K$.

What we have proved is the following lemma:

\begin{lemma} \label{lem1}
Let $(M,d)$ be a  metric space and  let $(E,\rho)$ be a complete metric space. If $T:M \to E$ is  a $K$-Lipschitz map, then there exists a complete metric space $(M_T, \overline d_T)$, an inclusion/quotient $1$-Lipschitz map $j:M\to M_T$ and a $K$-Lipschitz map $\overline{T}:M_{T} \to E$ such that $T=\overline T\circ j$, i.e., the following diagram commutes 
\[
\xymatrix{
M  \ar[rr]^{ T} \ar@{->}[dr]_{j} & & E.\\
& {M_{T}} \ar[ur]_{\overline{T}} & }
\]
Here, $\overline d_T$ is the metric associated to the pseudo metric
$$
d_T(x,y):= \frac{1}{K} \rho( T(x), T(y)), \quad x,y \in M.
$$
\end{lemma}

 Notice that the usual completion $\overline M$ of the metric space $M$ also satisfies the lemma, and it is usually seen as the smallest completed space that contains $M$ as a dense subspace. However, we look for a maximal completed space that contains (a quotient of) $M$ as a dense subspace. This is why the above construction is necessary to obtain the maximal factorization, as showed in the next theorem. 

\begin{theorem} \label{lem2}
	In the same setting and with the same notation as in Lemma \ref{lem1}, we have that the factorization through $M_T$ and the Lipschitz operator
	$\overline{T}$ is maximal in the following sense. If there is another complete  metric space $(J, \overline \rho)$ such that
	
	\begin{itemize}
		
		\item[i)] the operator $T$ can be factored as $T= T_0 \circ i_0$, where $T_0:J \to E$ is a $K$-Lipschitz operator and $i_0:M \to J$  is a $1$-Lipschitz inclusion/quotient map,
		
		\item[ii)] $i_0(M)$ is dense in $J$.

	\end{itemize}
	then there is a $1$-Lipschitz inclusion/quotient map $i:J \to M_T$ satisfying
	
	\begin{itemize}
		\item [1)] $i\circ i_0 = j$, and
		
		\item [2)] $\overline T\circ i = T_0$.
	\end{itemize}
\end{theorem}

\begin{proof}
Let us prove that there exists  $i:J\to M_T$. Let $z\in J$. Then by ii) there is a sequence $(x_n)\subset M$ such that $z = \overline \rho - \lim i_0(x_n)$. Therefore $(i_0(x_n))\subset J$ is a $\overline \rho$-Cauchy sequence. Note that
\begin{eqnarray*}
\overline d_T(j(x),j(y))= d_T(x,y)
&=& \frac{1}{K} \rho (T(x),T(y))\\
=\frac{1}{K} \rho (T_0\circ i_0(x),T_0\circ i_0(y))
&\leq&\overline \rho(i_0(x),i_0(y)),\quad x,y\in M.
\end{eqnarray*}
Thus $(j(x_n))$ is a $\overline d_T$-Cauchy sequence, and so there exists $w\in M_T$ such that $w = \overline d_T-\lim j(x_n)$. Define $i:J\to M_T$ by $i(z)= w$, where $z=\overline \rho-\lim i_0(x_n)$ and $w=\overline d_T-\lim j(x_n)$. We claim that $i$ is well defined. Indeed, if there is another sequence $(y_n)
\subset M$ such that 
\[
z= \overline \rho-\lim i_0(x_n) =\overline \rho-\lim i_0(y_n),
\]
then the sequence
\[
(i_0(x_1),i_0(y_1),i_0(x_2),i_0(y_2),\ldots)
\]
converges to  $z$ in $J$. Thus it is a $\overline \rho$-Cauchy sequence and we have that
\[
(j(x_1),j(y_1),j(x_2),j(y_2),\ldots)
\]
is a $\overline d_T$-Cauchy sequence. Therefore $\overline d_T-\lim j(x_n)=\overline d_T-\lim j(y_n)$. 

Let us show now $1)$. Let $x\in M$. Trivially we have $i_0(x) = \overline \rho - \lim i_0(x)$. Thus
\[
i(i_0(x))=\overline d_T-\lim j(x) = j(x).
\]
Since the inequality
\[
\overline d_T(j(x),j(y))\leq \overline \rho(i_0(x),i_0(y)),\quad x,y\in M
\]
holds, it follows that $i$ is $1$-Lipschitz.

To prove part  $2)$, let $z = \overline \rho-\lim i_0(x_n)$. Then
\begin{eqnarray*}
\overline T(i(z)) &= & \overline T(\overline d_T-\lim j(x_n))
= \rho-\lim \overline T\circ j(x_n)\\
&=& \rho-\lim T(x_n)
= \rho-\lim T_0\circ i_0(x_n)\\
&=&T_0(\overline \rho -\lim i_0(x_n))
=T_0(z),
\end{eqnarray*}
and so the result holds.
\end{proof}

\section{Extension of Lipschitz maps with values in  function spaces} \label{S3}

In this section we consider the problem of extending maps defined on metric spaces and that take values
in  lattices of functions  defined on a measure space $(\Omega,\Sigma,\mu)$; that is, spaces of classes of $\mu$-a.e equal functions. It must be said that the techniques used here are of different nature than the ones used by Kalton and the other authors that considered the problem for the case of ``$\infty$-type" spaces as explained in the introduction. The reason is that the structure of the spaces of integrable functions is meaningfully different to the case of $C(K)$ spaces. Broadly speaking, pointwise domination and norm domination are equivalent properties in the case of spaces of continuous functions, but this  is not the case for  spaces of measurable functions. This forces to adapt the definitions and to create new settings specifically constructed for the case of lattices of integrable functions.  Let us introduce some concrete definitions. We will assume that a set of functions $L$ is a lattice if
for every pair $f,g \in L,$ we have that $f \vee g \in L.$
 A subset $F(\mu)$ of $L^0(\mu)$ is a \textit{metric function space} if it is a metric space which is a lattice, and for every $f \in F(\mu)$ and $A \in \Sigma,$ $f \chi_A \in F(\mu).$ 

Our aim is to  analyze the extension procedure when maps with values in a metric function space $F(\mu)$ satisfy a Lipschitz type inequality. In this concrete setting of $F(\mu)$-valued mappings, we will consider pointwise $\mu$-a.e. variants of the inequality that defines a Lipschitz map. 

\begin{definition} If $(M,d)$ is a metric space, a map
$T: M\to F(\mu)$ is \textit{ pointwise $K$-Lipschitz $\mu$-a.e.} if
	\[
	\big| T(x)-T(y)  \big| \leq K \, d(x,y)  \quad \text{$\mu$-a.e.}
	\]
for all $x,y \in M.$ As usual, $K$ is supposed to be the infimum of all the constants satisfying the inequality. 
\end{definition}
To be more specific, 
 for each   $x,y \in M$ there is a set $A_{x,y}\in \Sigma$ with $\mu(\Omega\setminus A_{x,y})=0$ and 
$$
\big| T(x)(w) - T(y)(w) \big| \le K \, d(x,y), \mbox{ for all } w \in A_{x,y}.
$$

Dealing with classes of $\mu$-a.e. equal functions instead of functions leads to extra difficulties. It is easy to see that in this case the extension formulae (\ref{forext2}) and (\ref{forext3})  do not give necessarily  measurable functions. This is why we will have to limit the result to Lipschitz maps restricted to countable subsets.

\begin{theorem} \label{pointTh} Let $(M,d)$ be a metric space and let $S\subset M$ be a countable set. Let $F(\mu)$  be a metric 
function space  that
  is closed under translations defined by constants (that is, $f+a\in F(\mu)$ whenever $f\in F(\mu)$ and 'a' is a real constant).
If $T: S\to F(\mu)$  is a pointwise $K$-Lipschitz $\mu$-a.e. map, then there exists an extension $\hat T:M\to F(\mu)$ of $T$ that is pointwise $K$-Lipschitz $\mu$-a.e.

\end{theorem}

\begin{proof}
	By the Axiom of Choice, there exists an element $R\in \Pi_{x\in S}T(x),$ $R=(r_x)_{x\in S},$ such that for every $x$, $r_x$ is a measurable function that belongs to the equivalence class (determined by) $T(x)$. 
For each $y\in X$ consider the function 
	\begin{equation}\label{repres}
	\hat T(y)(w):= \displaystyle \sup_{x\in S}\{r_x(w)-Kd(x,y)\},\quad w\in \Omega
	\end{equation}
and  define $\hat T(y)(w)$ as its equivalence class 
(as usual in Measure Theory we write $\hat T(y)(w)$ for both, the class and a representative of the class).
Notice that under the assumption that $S$ is countable, the expression in (\ref{repres}) determines a measurable function. Indeed, for a fixed $y \in M$ we can consider the countable
set of functions $S_y=\{r_x(w)-Kd(x,y)\}.$ Of course, the supremum of such a set is a measurable function, and so $\hat T(y)$ is well defined. Let us divide the proof in two steps.

\textit{Step 1.} Let us show   that $\hat T$ extends $T$, i.e.,  if $y\in S$ then $\hat T(y)$ defined in (\ref{repres}) coincides with $T(y)$ $\mu$-a.e.
		Note that for every $x \in S$ there exists a measurable null set  $N_x(y)$ such that
	\[
	|T(y)(w)-T(x)(w)|\leq K d(x,y),\quad w\in \Omega\backslash N_x(y).
	\]
	Let $N(y):= \cup_{x\in S} N_x(y).$ Note that it is a measurable null set as  $S$ is countable. Note also that $T(x)$ is just defined $\mu$-almost everywhere for every $x \in S$ and so its value is not uniquely determined. For every $x\in S,$ consider any other element $t_x$ of the class  $T(x)$, and take the set
	\[
	M_x = \{w\in \Omega: t_x(w)\neq r_x(w)\}.
	\]
	It is clearly is a measurable null set, and so  $M = \cup_{x\in S}M_x$ is also a  measurable null set. If $w\in \Omega\setminus (N(y)\cup M)$, then
	\[
	t_y(w)\geq t_x(w)-Kd(x,y) = r_x(w)-Kd(x,y).
	\]
	Consequently, for such a $w$ we have
	\[
	t_y(w)\geq \sup_{x\in S}\{t_x(w)-Kd(x,y)\} = \sup_{x\in S}\{r_x(w)-Kd(x,y)\}.
	\]
	Since $y\in S$, we have $T(y)(w)  = r_y(w)=r_y(w)-Kd(y,y)$ $\mu$-a.e. Therefore ---again using an abuse of notation identifying the function with its equivalence class---, we get
	\[
	\sup_{x\in S}\{r_x(w)-Kd(x,y)\}=T(y)(w).
	\]

\textit{Step 2.}  Let us now see that $\hat T(x)\in F(\mu)$ for all $x\in M$. Let $x\in M$ and $y\in S$. We claim that
	\[
	|\hat T(x)(w)-T(y)(w)|\leq Kd(x,y)\quad \text{$\mu$-a.e.}
	\]
	Indeed, taking any representatives for $\hat T(x)$ and $T(y),$ ---we again write $\hat T(x)$ and $T(y)$ for them---, for all $w$ except in a $\mu$-null set we have
	\begin{eqnarray*}
		|\hat T(x)(w)-T(y)(w)|&=&|\sup_{z\in S}\{r_z(w)-K d(z,x)\}-\sup_{z\in S}\{r_z(w)-K d(z,y)\}|\\
		&\leq&\sup_{z\in S}K|d(z,x)-d(z,y)| \leq Kd(x,y)\quad \text{$\mu$-a.e.}
	\end{eqnarray*}
	Therefore,
	\[
	-Kd(x,y)\leq \hat T(x)(w)-T(y)(w)\leq Kd(x,y).
	\]
	Since $y\in S$, $T(y)\in F(\mu)$. Define $h_1(w) = T(y)(w)-Kd(x,y)$ and $h_2(w) = T(y)(w)+Kd(x,y)$, we have
	\[
	h_1(w) \leq \hat T(x)(w)\leq h_2(w).
	\]
	Consequently,
	\[
	|\hat T(x)(w)|\leq h(w)
	\]
	where $h(w) = (h_1 \vee h_2)(w)= \max\{h_1(w),h_2(w)\}$. Since $F(\mu)$ is a lattice and $h\in F(\mu)$, we conclude that  $\hat T(x)\in F(\mu)$ for all $x\in M$.

Since we have already proved in Step 1 that $\hat T(y) = T(y)$ for all $y\in S$ we get that $\hat T$ is a extension of  $T$.
	
	Step 3. Finally, let us show that $\hat T$ is pointwise $K$-Lipschitz $\mu$-a.e. Indeed,  if $x,y\in M$, then
	\begin{eqnarray*}
		|\hat T(x)(w)-\hat T(y)(w)|&=&|\sup_{z\in S}\{r_z(w)-Kd(z,x)\}-\sup_{z\in S}\{r_z(w)-Kd(z,y)\}|\\
		&\leq&\sup_{z\in S}K|d(z,x)-d(z,y)| \leq Kd(x,y)\quad \text{$\mu$-a.e.}
	\end{eqnarray*}
This finishes the proof.
\end{proof}

\begin{remark}  \label{remLp}

Let us write some examples to which the result above can be applied. Assume that we have a pointwise $K$-Lipschitz map $\mu$-a.e. $T$ from a countable subset $S$ of a metric space $(M,d)$ into a Banach function space $Y(\mu).$  Note that, in particular, Banach function spaces are metric function spaces.
 Then by Theorem \ref{pointTh} we have an extension $\hat T$ to the whole
space $M$ that is pointwise $K$-Lipschitz $\mu$-a.e. So, 
\begin{equation}\label{ineqdef}
|T(x)-T(y)|\leq Kd(x,y), \ \ \ \mu-\mbox{a.e.}
\end{equation}
for all $x,y\in M$. Straightforward arguments show that this implies the following Lipschitz type properties, depending on who is the space $Y(\mu).$

\begin{itemize}

\item[(1)] If $Y(\mu)=L^1(\mu)$ and $\mu$ is a finite measure, then the $\mu$-a.e. domination of $T$ given in (\ref{ineqdef})  implies that
	\[
	\|\hat T(x)- \hat  T(y)\|_{L^1(\mu)} = \int_\Omega | \hat  T(x)- \hat  T(y)| d \mu \le K\mu(\Omega)d(x,y),
	\]
	for all $x,y\in M$, i.e., $ \hat  T:M\to L^1(\mu)$ is a $K'$-Lipschitz operator with constant $K'\leq K \mu(\Omega)$.
	
\item[(2)] If $Y(\mu)=L^{\infty}(\mu)$ and $\mu$ is a $\sigma$-finite measure, then (\ref{ineqdef}) gives that 
	\[
	\| \hat  T(x)- \hat  T(y)\|_{L^{\infty}(\mu)}\leq K d(x,y),
	\]
	for all $x,y\in M$, i.e., $ \hat  T:M\to L^{\infty}(\mu)$ is a $K$-Lipschitz map.
	
\item[(3)] If $Y(\mu)=L^{p}(\mu)$ and $\mu$ is a finite measure, then
	\[
	\| \hat  T(x)- \hat  T(y)\|_{L^{p}(\mu)}\leq K \mu(\Omega)^{1/p} d(x,y),
	\]
	for all $x,y\in M$, i.e., $\hat  T:M\to L^{p}(\mu)$ is a  $K'$-Lipschitz operator with constant $K'\leq K \mu(\Omega)^{1/p}$.

\end{itemize}
\end{remark}

\begin{corollary}  
Let $S$ be a countable subset of a metric space $(M,d)$   and let $\mu$ be a finite measure. Let $1 \le p \le \infty.$ Consider a  pointwise $K$-Lipschitz $\mu$-a.e.  map $T:S \to L^p(\mu).$

 Then $T$ can be extended to $M$ in such a way that the extension $\hat T$ factors through an inclusion/quotient $1$-Lipschitz map $j$ and a $K'$-Lipschitz map $\overline{T}:M_{{\hat T}} \to L^p(\mu) $ with constant $K' \leq \mu(\Omega)^{1/p} K$, as
\[
\xymatrix{S \hookrightarrow
M \,\,\,\, \ar[rr]^{ \hat T} \ar@{->}[dr]_{j} & & L^p(\mu),\\
& {M_{\hat T}} \ar[ur]_{\overline{T}} & }
\]
where $(M_{\hat T}, \overline d_{\hat T})$ is a complete  metric  space in which $j(M)$ is dense, and $\overline d_{\hat T}$ is the metric associated to the pseudo metric
$$
d_{\hat T}(x,y):= \frac{1}{\mu(\Omega)^{1/p}  K} \, \big\| {\hat T}(x) - {\hat T}(y)\big\|_{L^p(\mu)}, \quad x,y \in M.
$$
Moreover, this extension/factorization is optimal in the sense of Theorem  \ref{lem2}.

\end{corollary}
\begin{proof}
If $T: S\to L^p(\mu)$  is a pointwise $K$-Lipschitz $\mu$-a.e. map, we have by Theorem \ref{pointTh}
that there exists an extension $\hat T:M\to L^p(\mu)$ of $T$ that is pointwise $K$-Lipschitz $\mu$-a.e.
Remark \ref{remLp} gives that $\hat T$ is also $\mu(\Omega)^{1/p} K$-Lipschitz, and so we can apply
Lemma  \ref{lem1} and Theorem \ref{lem2} to obtain the result.

\end{proof}

Let us mention here some particular facts regarding the reference papers on spaces with an $\infty$-type metric structure that we have commented in the introduction, that is $C(K)$-spaces, $\ell^\infty,$ $c_0$...
Concretely, for the case of discrete measure spaces the simplicity of the structure of measurable sets provides direct ways of proving the existence of extension for Lipschitz maps. In this case, $\mu$-a.e. properties and pointwise properties coincide, and we do not need any assumption on the cardinality of the set $S$. Also, there are no restrictions regarding finiteness ($\sigma$-finiteness) of the measure. Also, pointwise domination and norm domination are the same things in these spaces. Both facts together allow to identify the $\mu$-a.e. Lipschitz type property used in Theorem \ref{pointTh} and the classical Lipschitz property. Let us finish this section by illustrating these facts.

\begin{remark}
	Let $(M,d)$ be a metric space and $I$ an index set. Let $T:S\to \ell^{\infty}(I)$ be a $K$-Lipschitz map, where $S\subset M$. Then there exists an extension $\hat{T}:M\to \ell^{\infty}(I)$ of $T$ that is also a  $K$--Lipschitz map.
\end{remark}
\begin{proof}
Let us write $(T_i)_{i\in I}$ for the ``coordinate decomposition" of $T$. Then
	\[
	|T_i(x)-T_i(y)|\leq \|T(x)-T(y)\|_{\infty}\leq  K d(x,y),
	\]
	for all $x,y\in S$ and all $i\in I$. Consequently $T_i:S\to \mathbb{R}$ is a Lipschitz map with Lipschitz constant $Lip(T_i)\leq K$. By the McShane--Whitney Theorem, there exists a Lipschitz map $\hat{T}_i:M\to \mathbb{R}$  such that $\hat{T}_i|_{S}\equiv T_i$ and $Lip(\hat{T}_i) = Lip(T_i)\leq K$. Fix $x_0\in S$ and let $x\in M$. Then
	\[
	|\hat{T}_i(x)|\leq |\hat{T}_i(x)-\hat{T}_i(x_0)|+|T_i(x_0)|\leq K d(x,x_0)+\|T(x_0)\|_{\ell^\infty(I)}.
	\]
	Thus $(\hat{T}_i(x))_{i\in I}\in \ell^{\infty}(I)$. Define $\hat{T}:M\to\ell^{\infty}(I)$ by $\hat{T}(x) = (\hat{T}_i(x))_{i\in I}$. Note that $\hat{T}|_{S}\equiv T$ and
	\[
	|\hat{T}_i(x)-\hat{T}_i(y)|\leq K d(x,y),
	\]
	for all $x,y\in M$ and all $i\in I$. We have shown that $\hat{T}$ is a $K'$--Lipschitz map with constant $K'\leq K$.  However $K=K'$ as $\hat T$ is an extension of $T$.
\end{proof}

%

\section{ Measure-type domination and extension properties for Bfs-valued Lipschitz maps  } \label{S4}

In this section we study extension of maps satisfying  Lipschitz-type dominations
that involve measure-type theoretical elements. We restrict our attention to the case of operators whose range is a Banach function space, a particular case of the metric function spaces that have been considered in the previous section. In order to formalize the measure theoretical notions that are needed, we will introduce a set function $\phi$ in the domination equation.
Let $(\Omega, \Sigma, \mu)$ be a finite measure space and let $Y(\mu)$ be a Banach function space. Consider a subset $S$ of a metric space $(M,d).$


\begin{definition}
Let $\phi: \Sigma \to \mathbb R^+$ be  an increasing  bounded set function, that is, a set function satisfying that
for every $A,B \in \Sigma$ such that $B \subseteq A$, $\phi(B) \le \phi(A),$ and
$\sup_{A \in \Sigma} \phi(A) < \infty.$ 
We say that a Lipschitz map $T:M \to Y(\mu)$ is $\phi$-Lipschitz if
$$
\| (T(x)-T(y)) \chi_A \|_{Y(\mu)} \le  \phi(A) d(x,y) 
$$
for all $x,y \in M$ and all $A\in \Sigma$.
\end{definition}

The main examples in this section are related to functions $\phi$ that are given by norms of Banach function spaces $Z(\mu)$ over the same  measure $\mu.$ That is, we will consider functions $\phi$ as
$$
\phi(A):= K \| \chi_A\|_{Z(\mu)}, \quad A \in \Sigma,
$$
for a constant $K>0,$ that is supposed to be minimal as usual.
We will  say in this case that  \textit{$T$ is a $Z$-Lipschitz map with constant $K$ from $M$ to $Y(\mu)$ }. In particular, if $Z(\mu)=L^1(\mu)$ and $\phi(A)=\|\chi_A\|_{L^1(\mu)}=\mu(A)$, $A\in \Sigma$, we will say that $T:M\to Y(\mu)$ is $\mu$-Lipschitz with constant $K$ if there is $K'>0$ such that 
$$
\|(T(x)-T(y))\chi_A\|_{Y(\mu)}\leq K'\|\chi_A\|_{L^1(\mu)} d(x,y)
$$
for all $x,y\in M$ and all $A\in \Sigma$, and $K$ is the infimum of such constants.

Note that $\phi$ is bounded when the measure $\mu$ is finite, which will be  a natural assumption through this section.
We will need also some measure-related notions. 

\begin{definition}
Let $Y(\mu)$ be a Banach function space and $\nu:\Sigma \to \mathbb R$ a countably additive measure. We define the  $Y$-variation of $\nu$ by
$$
|\nu|_{Y}:=\sup_{\sum_{i=1}^n \alpha_i \chi_{A_i}  \in B_{Y'}}\Big| \sum_{i=1}^n \alpha_i \nu(A_i) \Big|.
$$
\end{definition}
The reader can find some applications of this notion for general vector measures in \cite[\S 4]{caldelsan} (see also \cite{blascalsan} and the references therein for more information about).

\begin{remark}
The space $L^1(\mu)$ can be endowed with the norm given by the semivariation of the measure that defines each function in it. Indeed, take $f \in L^1(\mu)$ and consider the measurable set 
$$
A_f=\{  w \in \Omega: f(w)>0\}
$$
and the set $A^c_f= \Omega \setminus A_f.$
Then
$$
\|f\|_{L^1(\mu)} = \int_{A_f} f d \mu - \int_{A^c_f} f d \mu \le 2 \max \Big\{ 
\int_{A_f} f d \mu,  \int_{A^c_f} - f d \mu \Big\} 
$$
$$
\le 2 \sup_{A \in \Sigma} 
\Big| \int_A f d \mu \Big| \le 2  \|f\|_{L^1(\mu)}.
$$
In the next result, we will use the equivalent norm $\| \cdot \|_{L^1(\mu),0},$
that is defined as
$$
\| f \|_{L^1(\mu),0} :=  \sup_{A \in \Sigma} 
\Big| \int_A f d \mu \Big|, \quad f \in L^1(\mu). 
$$
We have just proved that $\| \cdot \|_{L^1(\mu)} \le 2 \| \cdot \|_{L^1(\mu),0} \le 2\| \cdot \|_{L^1(\mu)}.$
\end{remark}

Recall that, if $\nu$ is a finite (real) measure, its semivariaton (in $\Omega$) is defined 
by 
$$
|\| \nu \||= \sup_{B \in \Sigma } |\nu(B)|,
$$
and its variation by $| \nu |= \sup_{g \in B_{L^\infty}(\nu)} |\int g d \nu|.$

\begin{proposition} \label{medL1} Let $M$ be a metric space and let $S$ be a subset of $M$.
Let $\phi: \Sigma \to \mathbb R^+$ be  an increasing and  bounded set function. Suppose that $T: S \to L^1(\mu)$  ---where $L^1(\mu)$ is endowed with the norm $\| \cdot\|_{L^1(\mu),0}$--  is a $\phi$-Lipschitz map. The following statements hold.
\begin{itemize}

\item[(1)] Suppose that for each $x \in M$  the set function  $\nu_x: \Sigma \to \mathbb R$ given by
$$
\nu_x(A):= \sup_{y \in S} \Big\{ \int_A T(y) d \mu - \phi(A) d(x,y) \Big\}, \quad A \in \Sigma,
$$
is a $\mu$-continuous  (countably additive) measure. Then $T$ admits a $\phi$-Lipschitz extension to $M$ into $L^1(\mu)$.

\item[(2)]  Conversely, if $T$ can be extended to all the space $M$ as a $\phi$-Lipschitz map, then for each $x \in M$
the set function
$$
\hat{\nu}_x(A):= \sup_{y \in M} \Big\{ \int_A T(y) d \mu - \phi(A) d(x,y) \Big\}, \quad A \in \Sigma,
$$
is a $\mu$-continuous  (countably additive) measure. Moreover, in this case
$$
|\|\hat \nu_x\|| = \|T(x)\|_{L^1(\mu),0}, \quad 
|\hat \nu_x| = \|T(x)\|_{L^1(\mu)}, \quad \text{and} \quad | \hat \nu_x(A) - \hat \nu_y(A)| \le \phi(A) d(x,y)
$$
for all $x,y \in M$ and $A \in \Sigma.$

\end{itemize}

\end{proposition}
\begin{proof}
Let us prove (1).  Let $x \in M$. Since $\nu_x$ is a (countably additive) measure that is $\mu$-continuous, we have by the Radon-Nikodym Theorem that there is a function $h_x \in L^1(\mu)$ such that
$\int_A h_x d\mu = \nu_x(A)$ for all $A \in \Sigma.$ Define $\hat T(x):= h_x$ for every $x \in M.$  Let us first see that $\hat T$ is an extension of $T$. Take $x\in S$. Clearly 
$$
\int_A T(x)\, d\mu\leq \nu_x(A)=\int_Ah_x\, d\mu, \ \ A\in \Sigma.
$$
To see the converse inequality, for any $y\in S$ we have 
$$
\Big| \int_A T(y) d \mu - \int_A T(x) d \mu \Big| \le 
\big\| (T(y)-T(x)) \chi_A \big\|_{L^1(\mu)} \le \,\,
\phi(A) d(x,y),
$$
and so
$$
\int_A T(y) d \mu - \int_A T(x) d \mu \le  \phi(A) d(x,y),
$$
what implies 
$$
\int_A T(y) d \mu - \phi(A)d(x,y) \le \int_A T(x) d \mu.
$$
Then,
$$
  \nu_x(A)  =  \sup_{y \in S} \Big\{ \int_A T(y) d \mu - \phi(A) d(x,y) \Big\} \le \int_A T(x) d \mu.
$$
Hence, from both inequalities,
$$
\int_A h_x\, d\mu=\int_A T(x)\, d\mu
$$
for all $A\in \Sigma$. Thus, $\hat T(x)=h_x=T(x)$ $\mu$-a.e., for all $x\in S$.

Now we need to prove that the extension $\hat T$, is $\phi$-Lipschitz.   Let $x,y \in M.$ Fix $A \in \Sigma$. Then,
$$
\big\|(T(x) - T(y)) \chi_A \big\|_{L^1(\mu),0} \le \sup_{\Sigma \ni B \subseteq A}
\big| \int_B \hat T(x) d \mu - \int_B \hat T(y)  d\mu \big|=
  \sup_{\Sigma \ni B \subseteq A}
\big| \nu_x(B)- \nu_y(B) \big|
$$
$$
=  \sup_{\Sigma \ni B \subseteq A} \Big(
  \sup_{z \in S} \Big\{ \int_B T(z) d \mu - \phi(B) d(x,z) \Big\} -
 \sup_{v \in S} \Big\{ \int_B T(v) d \mu - \phi(B) d(y,v) \Big\} \Big)
$$
$$
\le
 \sup_{\Sigma \ni B \subseteq A} \Big(
  \sup_{w \in S} \Big\{   \int_B T(w) d \mu - \phi(B) d(x,w)  -
  \int_B T(w) d \mu - \phi(B) d(y,w)     \Big\}  \Big)
$$
$$
=
 \sup_{\Sigma \ni B \subseteq A}   \,
\sup_{w \in S}   \, \phi(B)  (d(y,w)  -
   d(x,w) )   \le  \, \phi(A) \, d(x,y).
$$
Therefore, the extension given is $\phi$-Lipschitz too.

For part (2), fix $x \in M.$ Since we have that $T$ is defined in all $M,$ we also have that for $A \in \Sigma$
$$
\hat \nu_x(A) = \sup_{y \in M} \Big\{ \int_A T(y) d \mu - \phi(A) d(x,y) \Big\} \ge \int_A T(x) d \mu - \phi(A) d(x,x) = \int_A T(x) d \mu.
$$
On the other hand, for every $y \in M$ we have also that
$$
\Big| \int_A T(y) d \mu - \int_A T(x) d \mu \Big| \le \phi(A) d(x,y),
$$
and so
$$
\int_A T(y) d \mu - \int_A T(x) d \mu \le  \phi(A) d(x,y),
$$
what implies for all $y \in M$
$$
\int_A T(y) d \mu - \phi(A)d(x,y) \le \int_A T(x) d \mu.
$$
Therefore,
$$
\hat  \nu_x(A)  =  \sup_{y \in M} \Big\{ \int_A T(y) d \mu - \phi(A) d(x,y) \Big\} \le \int_A T(x) d \mu.
$$
Consequently, $\hat  \nu_x(A)= \int_A T(x) d \mu$ for all $A \in \Sigma$. A simple look to the definition of the semivariation  shows that $|\|\hat \nu_x\|| = \|T(x)\|_{L^1(\mu),0}.$  Since by hypothesis $T(x) \in L^1(\mu)$, we have indeed that $\nu_x$ is a
countably additive measure. Moreover,
$$
|\hat \nu_x |_{L^1(\mu)} =  \sup_{g \in B_{L^\infty(\mu)}} \left| \int g d \hat\nu_x\right|=
\sup_{g \in B_{L^\infty(\mu)}} \left| \int g T(x) d \mu\right| =
 \| T(x)\|_{L^1(\mu)}.
 $$
Finally, if $y \in M$ and $A \in \Sigma,$ we also have
$$
 | \hat \nu_x(A) - \hat \nu_y(A)|= \left| \int_A T(x) d\mu - \int_A T(y) d \mu\right|
\le \| (T(x)-T(y)) \chi_A\|_{L^1(\mu)} \le \phi(A) d(x,y).
$$

\end{proof}

\begin{remark}
The main situation where  Proposition \ref{medL1} can be applied, is given for functions $\phi$ defined by means of a norm associated to the space $L^1(\mu).$ There are two canonical cases.

\begin{itemize}

\item[(1)] The case $\phi_1(A):= K\mu(A),$ that is, $\phi_1(A):= K\|\chi_A\|_{L^1(\mu)},$ $A \in \Sigma,$ for some constant $K>0$; the extension provided by Proposition \ref{medL1} preserves the average variation, that is, the original $T$ satisfies that is $\phi_1$-Lipschitz, i.e., $T$ is $\mu$-Lipschitz, if and only if
$$
\sup_{A \in \Sigma, \, \mu(A) >0} \frac{\| (T(x)-T(y)) \chi_A \|_{L^1(\mu)}}{\mu(A)} \le K d(x,y), \quad x,y \in M.
$$

\item[(2)] The ``dual" case that consists of considering the function given by $\phi_\infty(A):= K \|\chi_A\|_{L^\infty(\mu)} =K,$
 for a constant $K>0,$ and for all $A \in \Sigma$ such that $\mu(A)>0,$ and $0$ otherwise. This gives the classical Lipschitz property. Indeed, $T$ is
$\phi_\infty$-Lipschitz if and only if
$$
\|T(x)-T(y)\|_{L^1(\mu)} =
\sup_{A \in \Sigma} \| (T(x)-T(y)) \chi_A \|_{L^1(\mu)} \le  \phi_\infty(A) d(x,y) =
K d(x,y),
$$
for $x,y \in M$.
That is, the original requirement for $T$ is that it is Lipzchitz with constant $K$. Note that for $x \in M,$ the function $\nu_x$ is given in this case by
$$
\nu_x(A) = \sup_{y \in S} \Big\{ \int_A T(y) d \mu - K d(x,y) \Big\}, \quad A \in \Sigma.
$$

\end{itemize}

\end{remark}

Let us show in the next result how we can adapt the extension given in Proposition \ref{medL1} for operators having values in a general Banach function space $Y(\mu).$

\begin{theorem}  \label{propmu}
Let $\mu$ be a finite measure.
Let $Y(\mu)$ be an order continuous Banach function space such that simple functions are dense in its dual. Let $(M,d)$ be a metric space and $S \subset M$
 and let $T: S \to Y(\mu)$ be a $\mu$-Lipschitz map with constant $K$. Suppose that
for every $x \in M,$  the set function  $\nu_x: \Sigma \to \mathbb R$ given by
$$
\nu_x(A):= \sup_{y \in S} \Big\{ \int_{A} T(y) d \mu - K \mu(A) d(x,y) \Big\},
$$
is a $\mu$-continuous  (countably additive) measure with finite $Y$-variation.
  Then $T$ admits a $Y$-Lipschitz extension with constant $K$ to $M$ into $Y(\mu)$.


\end{theorem}

\begin{proof}
 By the Radon-Nikodym theorem, there is an integrable function $h_{x}$ such that $\nu_x(A)= \int_A h_{x} d \mu$ for every $A \in \Sigma.$ Since  $\nu_x$ has finite $Y$-variation and taking into account that
$Y(\mu)$ is order continuous and simple functions are dense in the dual,
we have that
$$
|\nu_x|_{Y}=\sup_{\sum_{i=1}^n \alpha_i \chi_{A_i}  \in B_{Y'}}\Big| \sum_{i=1}^n \alpha_i \nu_x(A_i) \Big|
= \sup_{g \in B_{Y'}} | \int h_x g d \mu |  < \infty.
$$
Therefore, $h_x \in Y(\mu)$ and $\|h_x\|_{Y(\mu)}= |\nu_x|_{Y}$.  Since $\mu$ is a finite measure, $Y(\mu)$ is contained in $ L^1(\mu)$. Then, as $T:S\to\L^1(\mu)$ is $\mu$-Lipschitz with constant $K$, we can apply Proposition \ref{medL1} (1) (taking $\phi(A)=K\mu(A)$, $A\in \Sigma$)  and consider the extension $\hat T(x):= h_x$ for every $x \in M$ given in its proof. Let us prove that $\hat T$  is $Y$-Lipschitz with constant $K.$ Let $x,y \in M.$ Then
$$
\|(T(x) - T(y)) \chi_A\|_{Y(\mu)} = \sup_{g \in B_{Y'}} | \int_A (h_x -h_y) g d \mu |
$$
\begin{equation*} \label{eqYLip}
= \sup_{g=\sum_{i=1}^n \alpha_i \chi_{A_i }  \in B_{Y'}}\Big| \sum_{i=1}^n \alpha_i \big( \nu_x(A_i \cap A) - \nu_y(A_i\cap A) \big) \Big|.
\end{equation*}

If we fix  a norm one function $g=\sum_{i=1}^n \alpha_i \chi_{A_i}  \in B_{Y'},$ we have that
$$
\sup_{z\in S}\left(\int_{A_i\cap A} \hspace{-0.5cm}T(z)\, d\mu-K \mu(A_i\cap A) d(x,y)\right)-\sup_{v\in S}\left(\int_{A_i\cap A}\hspace{-0.5cm} T(v)\, d\mu-K \mu(A_i\cap A) d(y,v)\right)
$$
$$
\leq \sup_{w\in S} K\mu(A_i\cap A)(d(y,w)-d(x,w))\leq K\mu(A_i\cap A)d(x,y).
$$
Then, interchanging the role of $x$ and $y$ we get 
$$
\sup_{z\in S}\left(\int_{A_i\cap A} \hspace{-0.5cm}T(z)\, d\mu-K \mu(A_i\cap A) d(x,y)\right)-\sup_{v\in S}\left(\int_{A_i\cap A}\hspace{-0.5cm} T(v)\, d\mu-K \mu(A_i\cap A) d(y,v)\right)
$$
$$
\leq K\mu(A_i\cap A)d(x,y).
$$
Hence, 
$$
\Big| \sum_{i=1}^n \alpha_i \big( \nu_x(A_i \cap A) - \nu_y(A_i\cap A) \big) \Big|\leq 
 \sum_{i=1}^n |\alpha_i | \Big| \sup_{z \in S} \Big\{ \int_{A_i\cap A} T(z) d \mu - K \mu(A_i\cap A) d(x,z) \Big\}
$$
$$
 -
 \sup_{v \in S} \Big\{ \int_{A_i \cap A} T(v) d \mu - K \mu(A_i\cap A) d(y,v) \Big\} \Big|
$$
$$
\leq  \sum_{i=1}^n | \alpha_i | K
\mu(A_i\cap A) d(x,y)
= \|g \chi_A\|_{L^1(\mu)} K d(x,y)
$$
$$
 \hspace{-1cm}\le \|g\|_{Y'(\mu)} \|\chi_A\|_Y  K d(x,y) \le   K \|\chi_A\|_Y  d(x,y).
$$
Therefore, the extension given is $Y$-Lipschitz with constant $K$.

\end{proof}

%
%
%

\begin{remark}
Some information about the converse can  be also given in this case. Let us show that, \textit{ if $T:M \to Y(\mu)$ is a $Y$-Lipschitz map with constant $K$, then the set function $\nu_x$ given by
 $$
 \nu_x(A) = \sup_{y \in M} \Big\{ \int_A T(y) d \mu - K \|\chi_A\|_Y d(x,y) \Big\}, \quad A \in \Sigma,
 $$
 for every $x \in M$
 is a countably additive measure  with finite $Y$-variation.}  Assume for the aim of simplicity that $\|\chi_\Omega\|_{Y'} \le 1.$ Arguing as in the proof of Proposition \ref{medL1}, suppose that $T:M \to Y(\mu)$ is a $Y$-Lipschitz map with constant $K$.
Fix $x \in M.$ Then we clearly have that for $A \in \Sigma,$
$\nu_x(A)  \ge  \int_A T(x) d \mu.$
Taking into account that we are assuming that $\|\chi_\Omega\|_{Y'} \le 1,$ for every $y \in M$ we also have  that
$$
\Big| \int_A (T(y) - T(x)) \chi_\Omega d \mu \Big| \le \|(T(x)-T(y)) \chi_A\|_{Y} \le  K \|\chi_A\|_Y d(x,y),
$$
and so $\int_A T(y) d \mu - \int_A T(x) d \mu \le  K \|\chi_A\|_Y d(x,y).$
Therefore, we also obtain in this case that
$$
\nu_x(A)  =  \sup_{y \in M} \Big\{ \int_A T(y) d \mu - K \|\chi_A\|_Y  d(x,y) \Big\} \le \int_A T(x) d \mu.
$$
Consequently, $\nu_x(A)= \int_A T(x) d \mu$ for all $A \in \Sigma$. Since by hypothesis $T(x) \in Y(\mu) \subseteq  L^1(\mu)$, we have indeed that $\nu_x$ is a
countably additive measure. Finally, note that we have that
$$
| \nu_x | \le \sup_{g \in B_{Y'}} \int T(x) g d\mu \le \|T(x)\|_{Y(\mu)} < \infty,
$$
and so for all $x \in M,$ $\nu_x$ has finite $Y$-variation.

\end{remark}

\begin{corollary} 
Let $S$ be a subset of a metric space $(M,d)$   and let $\mu$ be a finite measure. 
Let $Y(\mu)$ be an order continuous Banach function space such that simple functions are dense in its dual. Let $T: S \to Y(\mu)$ be a $\mu$-Lipschitz map with constant $K,$ and suppose that
for every $x \in M,$  the set function  $\nu_x: \Sigma \to \mathbb R$ given by
$$
\nu_x(A):= \sup_{y \in S} \Big\{ \int_{A} T(y) d \mu - K \mu(A) d(x,y) \Big\},
$$
is a $\mu$-continuous  (countably additive) measure with finite $Y$-variation.

 Then $T$ can be extended to $M$ as a $K'$-Lipschitz map with constant $K'\leq K \| \chi_\Omega\|_Y$, and the extension $\hat T$ factors through an inclusion/quotient $1$-Lipschitz map $j$ and a $K \| \chi_\Omega\|_Y$-Lipschitz map $\overline{T}:M_{{\hat T}} \to Y(\mu) $ as
\[
\xymatrix{S \hookrightarrow
M \,\,\,\, \ar[rr]^{ \hat T} \ar@{->}[dr]_{j} & & Y(\mu),\\
& {M_{\hat T}} \ar[ur]_{\overline{T}} & }
\]
where $(M_{\hat T}, \overline d_{\hat T})$ is a complete  metric  space in which $j(M)$ is dense, and $\overline d_{\hat T}$ is the metric associated to the pseudo metric
$$
d_{\hat T}(x,y):= \frac{1}{K \| \chi_\Omega\|_Y} \, \big\| {\hat T}(x) - {\hat T}(y)\big\|_{Y(\mu)}, \quad x,y \in M.
$$
Moreover, this extension/factorization is optimal in the sense of Theorem  \ref{lem2}.

\end{corollary}
\begin{proof}
If $T: S\to Y(\mu)$  is a $\mu$-Lipschitz  map with constant $K$,  by Theorem \ref{propmu} we obtain
an extension $\hat T:M\to Y(\mu)$ of $T$ that is  $Y$-Lipschitz with constant $K$.
Then, it is easy to see that $\hat T$ is also $K \| \chi_\Omega\|_Y$-Lipschitz.
Lemma \ref{lem1} and Theorem \ref{lem2} give the result.

\end{proof}

\section{ Interpolation tools for obtaining Lipschitz-type inequalities for maps on Banach function spaces}  \label{S5}

To finish the paper, in this section we provide some interpolation tools for generalizing the Lipschitz-type inequalities that have been studied in the previous sections, in order to apply them in more general contexts. For example, as we have seen, some results hold  for operators on $L^\infty$ or $L^1$. Interpolated inequalities for interpolation spaces of these ``extreme cases" are easy to be obtained, as we will show in what follows. We will consider two cases: the lattice interpolation method ---that under some mild requirements coincides in the case of Banach function spaces with the complex interpolation method---, and the real interpolation of function lattices.

Through this section we will consider a finite measure space $(\Omega, \Sigma, \mu)$ and
two bounded functions $\phi_i:\Sigma\to \mathbb R^{+},$  $i=0,1$. 

\subsection{Calder\'on-Lozanowskii interpolation of function lattices and Lipschitz-type maps}
Let us recall the  lattice interpolation construction. Let $Y_0(\mu)$ and $Y_1(\mu)$ be Banach function spaces in $L^0(\mu).$ Let $ 0 \le \theta \le 1.$  We define the set 
$$
Y_0(\mu)^{1-\theta} \, Y_1(\mu)^\theta
$$
$$
 :=\{x \in L^0(\mu): \text{there are functions} \, x_0 \in Y_0, \, x_1 \in Y_1 \, \text{such that} \, |x| \le |x_0|^{1- \theta} |x_1|^\theta \}.
$$
It is a linear space, that is complete with the norm
$$
\|x\|_{Y_0^{1-\theta} Y_1^\theta} := \inf \, \|x_0\|^{1- \theta} \, \|x_1\|^\theta,
$$
where the  infimum is computed over all dominations like the one above. In fact, it is a Banach function space over $\mu$ that clearly contains $Y_0(\mu) \cap Y_1(\mu) .$

For the following result, consider  the inclusion maps $ I_\theta: Y_0(\mu) \cap Y_1(\mu) \hookrightarrow Y_0^{1-\theta} Y_1^\theta,$ for $\theta \in [0,1].$

\begin{proposition}\label{propint}
	Let $Y_0(\mu)$ and $Y_1(\mu)$ be Banach function spaces and let $M$ be a metric space. Consider
a map $T:M \to Y_0(\mu) \cap Y_1(\mu)$ such that  there are constants $K_0, K_1 >0$ such that  the maps $T_i:=I_i \circ T:M\to Y_i(\mu),$ $i=0,1,$
satisfy the Lipschits inequalities
$$
\|T_i(x)-T_i(y)\|_{Y_i(\mu)}\leq K_i \, d(x,y), \quad x,y\in M.
$$

Then, for each $\theta\in (0,1)$ the map $T_{\theta}:M\to Y_0^{1-\theta}Y_1^{\theta}$ given by $T_{\theta}:=  I_\theta \circ T$ is well defined and satisfies
	\[
	\|T_{\theta}(x)-T_{\theta}(y)\|_{{Y_0^{1-\theta} Y_1^\theta}}\leq K_0^{1-\theta} K_{1}^{\theta}  \, d(x,y), \quad x,y \in M.
	\]
\end{proposition}

\begin{proof}

The proof is given by a simple computation. All the maps involved are clearly defined. Thus, taking into 
account that the elements $T(x)$ are defined as functions of $L^0(\mu)$ and so $T_0(x)=T_\theta(x)=T_1(x) \in L^0(\mu)$ for every $x \in M,$ we have that for $x,y \in M,$
$$
|T_{\theta}(x)-T_{\theta}(y)| 
= |T_0(x)-T_0(y)|^{1-\theta}|T_1(x)-T_1(y)|^{\theta},
$$
	and therefore
	\[
	\|T_{\theta}(x)-T_{\theta}(y)\|_{{Y_0^{1-\theta} Y_1^\theta}}\leq \|T_0(x)-T_0(y)\|_{Y_0(\mu)}^{1-\theta}\|T_1(x)-T_1(y)\|_{Y_1(\mu)}^{\theta}\leq K_0^{1-\theta} K_{1}^{\theta} \, d(x,y).
	\]	
\end{proof}

\begin{remark}
	The same computations as above provide also the following results for some of the cases that we have studied in the present paper. In particular, using the same notation that in Proposition \ref{propint}, we obtain
	\begin{itemize} 
	\item[(a)] For $\phi$-Lipschitz maps. Assume that  for $i=0,1,$ $\phi_i$ are increasing set functions such that 
	\[
	\|(T_i(x)-T_i(y))\chi_A\|_{Y_i(\mu)}\leq \phi_i(A) \, d(x,y) \quad A\in \Sigma, \,\, x,y\in M. 
	\]
	In this case, it can be easily seen that the map $T_{\theta}$ satisfies
	\[
	\|(T_{\theta}(x)-T_{\theta}(y))\chi_A\|_{{Y_0^{1-\theta} Y_1^\theta}}\leq  \phi_0(A)^{1-\theta}\phi_{1}(A)^{\theta} \, d(x,y).
	\]
	
	\item[(b)] For  the case of  pointwise $K$-Lipschitz $\mu$-a.e. maps we get a similar result. Assume that the following inequalities hold for $i=0,1,$
	\[
	|T_i(x)-T_i(y)|\leq K_i \, d(x,y) \quad \text{$\mu$-a.e.}, \quad x,y\in M.
	\]
	Then it can be easily seen that  the mapping $T_{\theta}$ satisfies
	\[
	|T_{\theta}(x)-T_{\theta}(y)|\leq  K_0^{1-\theta} K_1^{\theta}d(x,y)  \quad \text{$\mu$-a.e.}, \quad x,y\in M.
	\]
\end{itemize}
\end{remark}

\subsection{Real interpolation and Lipschitz-type inequalities} Some interpolation-type inequalities can also be obtained for our class of maps by using  real interpolation, concretely, the  $K$-functional. We consider here the case when $M$ is also a Banach space with the metric associated to its norm, in order to provide some results on $\phi$-Lipschitz maps in the Banach space setting.
	Let $\overline E=(E_0,E_1)$ be an interpolation couple of Banach spaces. For $t>0$, we consider the functional
	\[
	K(t,a,E_0,E_1) = \inf\{\|a_0\|_0+t\|a_1\|_1:a =a_0+a_1,\ a_0\in E_0,\ a_1\in E_1\}, \quad a \in E_0+E_1.
	\]
	Let $\theta\in (0,1)$ and $p\in [1,\infty)$. The interpolation space $\overline E_{\theta,p}$ is defined to be the set of all  $a\in E_0+E_1$ satisfying 
	\[
	\left(\int_{0}^{\infty}(t^{-\theta}K(t,a,E_0,E_1))^{p}\frac{dt}{t}\right)^{\frac{1}{p}}<\infty.
	\]
	The norm on $\overline E_{\theta,p}$ is given by
	\[
	\|a\|_{\theta,p}=	\left(\int_{0}^{\infty}(t^{-\theta}K(t,a,E_0,E_1))^{p}\frac{dt}{t}\right)^{\frac{1}{p}}.
	\]

Under some Lipschitz-type requirements for a map $T:E_1 \to Y_1(\mu),$ we obtain the following 
result, that provides a $\phi$-Lipschitz type inequality at least when $T$ is linear.

\begin{proposition}
	Let $E_0,E_1$ Banach spaces and $Y_0(\mu)$ and $Y_1(\mu)$ Banach function spaces such that $E_0\subset E_1$. Assume that the following  $\overline E =(E_0,E_1)$ and $\overline Y = (Y_0(\mu),Y_1(\mu))$ are Banach couples. Let $\phi_i: \Sigma \to \mathbb R^+$ be increasing set functions, $i=0,1.$ 
Suppose that $T:E_1\to Y_1(\mu)$ is a map satisfying the following properties. 

\begin{itemize}
	
\item[(a) ] $T(E_0)\subset Y_0$ and $\|T(x)\chi_A\|_{Y_0}\leq \phi_0(A)\|x\|_{E_0}$ for all $A\in \Sigma$ and $x\in E_0$, and

\item[(b)] $\|(T(x)-T(y))\chi_A\|_{Y_1}\leq \phi_1(A)\|x-y\|_{E_1}$ for all $A\in \Sigma$ and $x,y\in E_1$.

\end{itemize} 
	
Let $\theta\in (0,1)$ and $p\in [0,\infty]$.	Then $T:\overline E_{\theta, p}\to \overline{Y}_{\theta,p}$  is a well-defined map, and
	\[
	\|T(x)\chi_A\|_{\overline Y_{\theta,p}}\leq\phi_0(A)^{1-\theta}\phi_1(A)^{\theta}\|x\|_{\overline E_{\theta,p}} \quad x\in \overline E_{\theta,p} \text{ and }A\in \Sigma.
	\]
\end{proposition}

\begin{proof}
	Given $x\in E_1 = E_0+E_1,$ $t>0$ and $\varepsilon>0$, we choose $x_i\in E_i$ with $i=0,1$ such that
	\[
	\|x_0\|_{E_0}+t\frac{\phi_1(A)}{\phi_0(A)}\|x_1\|_{E_1}\leq (1+\varepsilon)K\left(t\frac{\phi_1(A)}{\phi_0(A)},x;E_0,E_1\right).
	\]
	Writing $T(x)\chi_A = T(x_0)\chi_A+(T(x)-T(x_0))\chi_A$, we have that
	\begin{eqnarray*}
		K\left(t,T(x)\chi_A;Y_0,Y_1\right) &\leq& \|T(x_0)\chi_A\|_{Y_0}+t\|(T(x)-T(x_0))\chi_A\|_{Y_1}\\
		&\leq& \phi_0(A)\|x_0\|_{E_0}+t\phi_1(A)\|x-x_0\|_{E_1}\\
		&\leq& \phi_0(A)(1+\varepsilon)K\left(t\frac{\phi_1(A)}{\phi_0(A)},x;E_0,E_1\right).
	\end{eqnarray*}
	Therefore,
	\[
	K\left(t,T(x)\chi_A;Y_0,Y_1\right) \leq \phi_0(A)K\left(t\frac{\phi_1(A)}{\phi_0(A)},x;E_0,E_1\right).
	\]
	Besides that,
	\[
	t^{-\theta}K\left(t,T(x)\chi_A;Y_0,Y_1\right)\leq \phi_0(A)^{1-\theta}\phi_1(A)^{\theta}\left(t\frac{\phi_1(A)}{\phi_0(A)}\right)^{-\theta}K\left(t\frac{\phi_1(A)}{\phi_0(A)},x;E_0,E_1\right).
	\]
	Consequently
	\[
	\|T(x)\chi_A\|_{\overline Y_{\theta,p}}\leq \phi_0(A)^{1-\theta}\phi_1(A)^{\theta}\|x\|_{\overline E_{\theta,p}}.
	\]
\end{proof}

\end{document}